\documentclass[11pt,reqno,tbtags]{amsart}

\usepackage{amsthm,amssymb,amsfonts,amsmath}
\usepackage{amscd} 
\usepackage{mathrsfs}
\usepackage[numbers, sort&compress]{natbib} 
\usepackage{subfigure, graphicx}
\usepackage{vmargin}
\usepackage{setspace}
\usepackage{paralist}
\usepackage[pagebackref=true]{hyperref}
\usepackage[usenames,dvipsnames]{color}
\usepackage{tcolorbox}
\usepackage[normalem]{ulem}
\usepackage{stmaryrd} 
\usepackage[refpage,noprefix,intoc]{nomencl} 

\providecommand{\R}{}

\providecommand{\N}{}

\renewcommand{\R}{\mathbb{R}}

\renewcommand{\N}{{\mathbb N}}

\newcommand{\p}[1]{{\mathbf P}\left\{#1\right\}}

\newcommand{\I}[1]{{\mathbf 1}_{[#1]}}

\newcommand\cS{{\mathcal S}}
\newcommand\cT{{\mathcal T}}

\newcommand{\eqdist}{\ensuremath{\stackrel{\mathrm{d}}{=}}}
\newcommand{\convdist}{\ensuremath{\stackrel{\mathrm{d}}{\rightarrow}}}

\providecommand{\ora}[1]{}
\renewcommand{\ora}[1]{\overrightarrow{#1}}
\DeclareRobustCommand{\SkipTocEntry}[5]{} 

\reversemarginpar
\definecolor{clou}{rgb}{0.8,0.25,0.5125}

\newtheorem{thm}{Theorem}

\newtheorem{prop}[thm]{Proposition}

\numberwithin{equation}{section}
\numberwithin{thm}{section}

\graphicspath{ {./images/} }
\usepackage{wrapfig}

\newcommand{\trees}{\mathscr{T}}
\newcommand{\sequence}[1]{\mathrm{#1}}

\newcommand{\dist}{\ensuremath{\mathrm{dist}}}
\newcommand{\rt}{\rho}
\newcommand{\forest}{\mathscr{F}}

\newcommand{\tree}{\mathrm{t}}
\newcommand{\randomtree}{\mathrm{T}}

\newcommand{\hyt}{\ensuremath{\mathrm{ht}}}

\begin{document}
	
\title{The Foata--Fuchs proof of Cayley's formula, and its probabilistic uses} 
\author[L. Addario-Berry]{Louigi Addario-Berry}
\address{Department of Mathematics and Statistics, McGill University, Montr\'eal, Canada}
\author[S. Donderwinkel]{Serte Donderwinkel}
\author[M. Maazoun]{Micka\"el Maazoun}
\author[J.B. Martin]{James B.\ Martin}
\address{Department of Statistics, Oxford, UK}

\email{louigi.addario@mcgill.ca}
\email{serte.donderwinkel@st-hughs.ox.ac.uk}
\email{mickael.maazoun@stats.ox.ac.uk}
\email{martin@stats.ox.ac.uk}

\date{November 17, 2022} 


\begin{abstract} 
We present a very simple bijective proof of Cayley's formula due to Foata and Fuchs (1970). 
This bijection turns out to be very useful when seen through a probabilistic lens; we explain some of the ways in which it can be used to derive probabilistic identities, bounds, and growth procedures for random trees with given degrees, including random $d$-ary trees.  We also introduce a partial order on the degree sequences of rooted trees, and conjecture that it induces a stochastic partial order on heights of random rooted trees with given degrees.
\end{abstract}

\maketitle
	\section{Introduction}\label{sec:intro} 
	
	A rooted tree is a triple $\tree=(V,E,\rt)$ where $(V,E)$ is a tree and $\rt \in V$. We write 
	\[
	\trees_n = \{\tree=(V,E,\rt): \tree\mbox{ is a rooted tree with }V=[n]\} 
	\]
	for the set of rooted trees with vertex set $[n]:=\{1,\ldots,n\}$. 
	{\em Cayley's formula} (which to the best of current knowledge was first established by \citet{MR1579119}) provides a very simple formula for $|\trees_n|$. 
	\begin{thm}[Cayley's formula]\label{thm:cayley}
		$|\trees_n| = n^{n-1}$. 
	\end{thm}
	There are numerous proofs of Cayley's formula in the literature. The 1967 survey by \citet{MR0214515} presents ten such proofs, including the proof via so-called {\em Pr\"ufer codes}, which is probably the one most frequently presented in undergraduate texts. More recent proofs include those discovered by \citet[Section 2.2]{MR633783}, which considers doubly-rooted trees; by \citet{MR1673928}, which analyzes a coalescent process for building rooted labeled trees; and by \citet[Section 3.6]{MR3617364}, which uses the connection between random labeled trees and conditioned Poisson branching processes. 
	
	In this work, we present a proof of Cayley's formula due to Foata and Fuchs \cite{FoataFuchs}, via a so-called {\em line-breaking} construction. We believe this is the simplest proof of Cayley's formula yet discovered, but it is not well-known. (In fact, in the first version of this manuscript we believed it to be new; we thank Adrien Segovia for pointing out the work \cite{FoataFuchs} to us.) 
	The reason it has been overlooked may be because the focus of the paper \cite{FoataFuchs} is on a bijection between the set of sequences $\{(x_1,\ldots,x_n): x_i \in [n],1 \le i \le n\}$ and the set of functions $f:[n]\to[n]$. The Foata--Fuchs bijection map yields a bijection between sequences with $x_1=x_2$ to $\cT_n$: given such a sequence $(x_1,\ldots,x_n)$, if 
	$f:[n] \to [n]$ is the resulting function, then the corresponding tree in $\trees_n$ has root $x_1$ and edge set $\{(i,f(i)),i \in[n]\setminus\{x_1\}\}$. For an English-language presentation of the full Foata--Fuchs bijection, see \cite{Broder}.
	
	We discuss related constructions for unrooted trees, rooted trees with marked vertices and rooted forests in  Section~\ref{sec:relatedconstr}.  The bijection is extremely useful for the analysis of random combinatorial trees; Section~\ref{sec:random} of this paper discusses some of its probabilistic implications. In that section, we also define an algorithm for growing random rooted trees with given vertex degrees. When applied to regular trees, the algorithm yields a sequence $(\mathrm{T}_m,m \ge 1)$, where for each $m \ge 1$, $\mathrm{T}_m$ is a uniformly random rooted, leaf-labeled $d$-ary tree with $m$ internal nodes, and $\mathrm{T}_{m+1}$ is generated from $\mathrm{T}_m$ by local regrafting. 
	Our algorithm is somewhat similar in spirit to R\'emy's algorithm \cite{remy}, which generates a sequence of uniform binary leaf-labeled trees.
	
	Before giving the proof of Cayley's formula, we introduce some terminology. 
	Suppose $\tree$ is a tree, $S$ is a connected subset of its vertices, and
	$x$ is a vertex.
	The \textit{path from $S$ to $x$ in $\tree$} is the unique path in $\tree$ which
	starts at a vertex of $S$, does not visit any other vertex of $S$, and ends at
	$x$. If $P$ is a path, we may also write $P$ to denote the set of vertices of the path.
	Finally, a \textit{leaf} of $\tree$ is a non-root vertex of $\tree$ with degree $1$.
	
	\begin{figure}%
		\begin{centering}
			\includegraphics[width=0.5\textwidth]{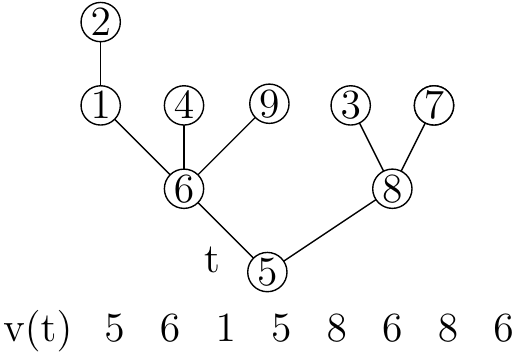}
			\caption{A tree $\tree$ and the corresponding sequence $\sequence{v}(\tree)$.}
			\label{mfig:cayley_bijection}
		\end{centering}
	\end{figure}
	
	\begin{proof}[Proof of Cayley's formula]
		The term $n^{n-1}$ counts sequences $\sequence{v}=(v_1,\ldots,v_{n-1})\in [n]^{n-1}$. We prove the theorem by introducing a bijection between $ \trees_n$ and $[n]^{n-1}$.
		\begin{tcolorbox}[title=Bijection]
			For a tree $\tree$ on $[n]$ with root $\rho$:
			\begin{itemize}
				\item Let $\ell_1<\ell_2<\dots<\ell_k$ be the leaves of $t$.
				\item
				Let $S_0=\{\rho\}$.
				\item
				Recursively, for $i=1,\dots, k$, let $P_i$ be the path
				in $\tree$ from $S_{i-1}$ to $\ell_i$, and let $S_i=S_{i-1}\cup P_i$.
				Let $P_i^*$ be $P_i$ omitting its final point.
				\item
				Let $\sequence{v}(\tree)$ be the concatenation of $P_1^*$, $P_2^*,\dots,P_k^*$.
			\end{itemize}
		\end{tcolorbox}
		We claim that $\sequence{v}(\tree)\in [n]^{n-1}$. Indeed, by definition, for each $i$, $P_i^*\subset [n]$. Moreover, note that the edge sets of $P_1$, $P_2,\dots,P_k$ form a partition of the edge set of $\tree$, and for each $i$, $|P_i^*|$ equals the number of edges in $P_i$. This implies that the length of $\sequence{v}(\tree)$ is equal to the number of edges in $\tree$, which equals $n-1$.
		
		To show that the above construction is a bijection, we describe its inverse.
		For a sequence $\sequence{v}=v_1, v_2, \dots, v_m$,
		we say that $j\in \{2,\dots, m\}$ is the location of a repeat of $\sequence{v}$ if 
		$v_j=v_i$ for some $i<j$.
		\begin{tcolorbox}[title=Inverse of the bijection]
			Given a sequence $\sequence{v}=(v_1, v_2, \dots, v_{n-1})\in [n]^{n-1}$:
			\begin{itemize}
				\item Let $j(0)=1$, let $j(1)<j(2)<\dots<j(k-1)$ be the locations of the repeats
				of the sequence $\sequence{v}$, and let $j(k)=n$.
				\item Let $\ell_1<\ell_2<\dots<\ell_k$ be the elements of $[n]$ not occurring
				in $\sequence{v}$.
				\item For $i=1,\dots, k$, let $P_i$ be the path
				$(v_{j(i-1)}, \dots, v_{j(i)-1}, \ell_i)$ with $j(i) - j(i-1)$ edges.
				\item Let $\tree(\sequence{v})$ be the graph with vertex set $[n]$, with root $v_1$
				and with edge set given by the union of the edges of the paths
				$P_1, P_2, \dots, P_k$.
			\end{itemize}
		\end{tcolorbox}
		We claim that $\tree(\sequence{v})$ is a rooted tree with vertex set $[n]$. 
		By construction, 
		for $i=2,\dots, k$, the set of vertices $\cup_{m<i} P_m$ and the path $P_i$
		intersect at a single point (namely $v_{j_{i-1}}$, the first point of $P_i$).
		Hence, by induction, for each $i$ the union of the edges of the paths
		$P_1, \dots, P_i$ is a tree, and in particular taking $i=k$ we have that $\tree(\sequence{v})$ 
		is a tree. 
		Since for $i=1,\dots, k$, the final point of $P_i$ is $\ell_i$, and since no element $\ell_i$ 
		appears as the first point of any other path $P_m$, we have that the leaves 
		of $\tree(\sequence{v})$ are 
		$\ell_1,\ldots, \ell_k$.
		
		We will argue that if $\sequence{v}\in [n]^{n-1}$, then $\sequence{v}(\tree(\sequence{v}))= \sequence{v}$. Let $k$, $P_1,\dots, P_k$, $\ell_1,\dots, \ell_k$, and  $v_1,v_{j_1},\dots, v_{j_k}$ be defined as in the inverse of the bijection applied to $\sequence{v}$.  By definition, $v_1$ is the root in $\tree(\sequence{v})$. 
		Note that $P_1$ is a path contained in $\tree(\sequence{v})$ with endpoints $v_1$ and $\ell_1$. The fact that  $\tree(\sequence{v})$ is a tree implies that $P_1$ is the unique path from $v_1$ to $\ell_1$ in $\tree(\sequence{v})$.
		Similarly, for $1<1\leq k$, $P_i$ is a path contained in $\tree(\sequence{v})$ with endpoints 
		$v_{j_{i-1}}$ and $\ell_i$, and with 
		$P_i\cap \cup_{m<i}P_m=\{v_{j_{i-1}}\}$, so that $P_i$ is the unique path 
		in $\tree(\sequence{v})$ from the union of $P_1,\dots,P_{i-1}$ to $\ell_i$. Hence  $\sequence{v}(\tree(\sequence{v}))$ is the concatenation of $P^*_1,\dots, P^*_k$, giving $\sequence{v}(\tree(\sequence{v}))= \sequence{v}$ as desired.
	\end{proof} 
	
	\section{Extensions and related constructions}\label{sec:relatedconstr}
	
	\subsection{Trees with given degrees.} 
	The bijection has the property that for a given tree $\tree$, in the associated sequence $\sequence{v}$ the number of times a given integer $k$ appears is precisely the number of children of the vertex with label $k$ in $\tree$.
	
	The {\em type} of a rooted tree $\tree$ is the vector $\sequence{n}=(n_c,c \ge 0)$, where $n_c$ is the number of vertices of $\tree$ with exactly $c$ children. 
	Writing $n = \sum_{c \ge 0} n_c$ and 
	\[
	\trees(\sequence{n}) = \{\tree \in \trees_n: \tree\mbox{ has type }\sequence{n}\}\, ,
	\] 
	then 
	\[
	|\trees(\sequence{n})| = {n \choose n_c,c \ge 0} \cdot \frac{(n-1)!}{\prod_{c \ge 0} (c!)^{n_c}}. 
	\]
	This identity (also found in \cite[Theorem 1.5]{MR1630413} and \cite[Corollary 3.5]{MR1676282}) follows from the specializations of the bijection to trees of a given type: the right-hand side counts the number of ways to first choose the labels of the vertices with $c$ children for each $c \ge 0$, then choose a sequence of length $n-1$ in which the label of each vertex with $c$ children appears exactly $c$ times.

	\subsection{Unrooted trees}  A simple trick allows us to directly study unrooted trees with this bijection by considering vertex $1$ as the root and removing the leading $1$ of the corresponding coding sequence.  This gives a proof of the more well-known form of Cayley's formula which states that there are $n^{n-2}$ unrooted trees of size $n$.
	
	There is another way to adapt the bijection to encode unrooted labeled trees of size $n$ by sequences in $[n]^{n-2}$. It consists in letting $P^*_1$ be the path between the first and second-lowest-labeled vertices of degree 1 (excluding both endpoints) and continuing as in the bijection for rooted trees.

	\subsection{Rooted trees with marked vertices}
	The bijection can also be modified to encode rooted labeled trees with $r\geq 1$ 
	distinguishable marks on the vertices, $(t,m_1,\ldots,m_r) \in \trees_n \times [n]^{r}$, by sequences in $[n]^{n+r-1}$. 
	The modification consists of changing the definition of $P_i^*$ in the recursive step slightly when constructing the sequence from the tree: 
	for $i=1,...,r$, $P_i^*$ is the path from $S_{i-1}$ to the $i$-th marked point (including the final point of this path), and for $i=1,\ldots,k$, $P_{r+i}^*$ is the path from $S_{r+i-1}$ to the $i$-th unmarked leaf $\ell_i$ (excluding its final point). Here the number of appearances of a vertex in the coding sequence is the sum of its number of children and its number of marks.
	
	By taking the $r\to\infty$ limit of this bijection and applying it to an i.i.d. sequence of elements of $[n]$, one recovers the construction of the so-called {\em p-trees} from \cite{MR1741774}.
	
	\subsection{Rooted forests}
	
	The bijection above also extends to forests. Given a set $S \subset [n]$, write $\forest_n^S$ for the set of forests $F$ with vertex set $[n]$ and root set $S$. Setting $s=|S|$, we describe a bijection between $\forest_n^S$ and the set 
	\[
	\{\sequence{v}=(v_1,\ldots,v_{n-s})\in [n]^{n-s}: v_1 \in S\}\, ,
	\]
	which has cardinality $sn^{n-s-1}$. The analogue of the ``inverse bijection'' is the easiest to describe, so we begin with that. To construct a sequence $(v_1,\ldots,v_{n-s})$ from $F \in \forest_n^S$, proceed just as in the above proof of Cayley's formula, but start from the $s$-vertex forest $F_0$ containing only the root vertices $S$, and at each step append the labels along the path to the smallest labeled leaf not already in the current forest (excluding the leaf itself). 
	
	Conversely, here is how to construct a forest $F \in \forest_n^S$ from a sequence $\sequence{v}=(v_1,\ldots,v_{n-s})$ with $v_1 \in S$. Say that $i$ is 
	the location of a repeat if $i > 1$ and either $v_i \in S$ or there is $1 \le j < i$ such that $v_i =v_j$. 
	Denote the locations of repeats of $\sequence{v}$ by $j(1),\dots, j(k-1)$ in increasing order, 
	and let $j(k)=n-s+1$.
	List the integers from $[n]\setminus S$ which do not appear in $\sequence{v}$ as $\ell_1,\ldots,\ell_{k}$ in increasing order.  
	
	Form a graph $F_S=F_S(\sequence{v})$ with vertices $[n]$, root set $S$, and edge set 
	\[
	\big\{v_iv_{i+1}: i \in [n-s-1], i+1 \not \in \{j(1),\ldots,j(k)\}\big\} \cup \{v_{j(i)-1}\ell_i,1 \le i \le k\}. 
	\]
	Essentially the same argument as in the proof of Cayley's formula shows that the connected components of $F_S$ are trees and that there are $s$ components of $F_S$, each containing exactly one vertex of $S$. Thus, rooting each component of $F_S$ at its unique element of $S$ turns $F_S$ into an element of $\forest_n^S$. 
	
	Since the construction is bijective, this yields the following  extension of Cayley's formula (which is in fact stated by both \citet{MR1579119} and \citet{cayley}): 
	\[
	|\forest_n^S|=|\{\sequence{v}=(v_1,\ldots,v_{n-s})\in [n]^{n-s}: v_1 \in S\}|= sn^{n-s-1}\, .
	\]
	This extension of the bijection can also be specialized to count forests with fixed types. 
	
	\subsection{Other coding sequences of trees}
	A variant of Pr\"ufer codes for rooted trees, first described by Neville \cite{MR54936}, gives another bijection between $\trees_n$ and $[n]^{n-1}$. It coincides with the original Pr\"ufer code if one sees vertex $n$ as the root and then removes the resulting trailing $n$. Deo and Micikevicius \cite{MR1887428} found that reading the rooted Pr\"ufer code in reverse leads to a simpler description of the bijection, which shares some features of the bijection presented in this paper.
	This bijection was independently discovered by Seo and Shin \cite{MR2353128} under the name \textit{Reverse-Pr\"ufer codes} to study leader vertices in trees, and also by Fleiner \cite{Fleiner}, 
	who remarked that it can be used to describe the law of the height of a uniformly sampled vertex in a random labeled tree (see Proposition \ref{prop:distance_vtx}, below). 
	
	\section{Random trees}\label{sec:random}
	The bijection above has numerous consequences for random trees. This section first explains how the bijection can be used to study typical and extreme distances in random trees, then discusses how it can be used to define growth procedures for random trees, somewhat in the spirit of R\'emy's algorithm \cite{remy} for growing uniformly random leaf-labeled binary trees. 
	
	For a finite set $\mathcal{S}$, we will write $X\in_u \cS$ to mean that $X$
	is chosen uniformly at random from the set $\cS$. 
	
	\subsection{Distances in random trees}
	For a tree $\tree=(V,E,\rt)$ and vertices $u,v \in V$, write $[u,v]=[u,v]_{\tree}$ for the unique path from $u$ to $v$ in $\tree$, and $\dist_\tree(u,v)$ for the distance from $u$ to $v$ in $t$, which equals the number of edges of $[u,v]$. Also, write $|v|$ for the graph distance from $\rt$ to $v$. 
	
	\begin{prop}\label{prop:leaf_identity}
		Let $\randomtree \in_u \trees_n$ and let $L$ be a uniformly random leaf of $\randomtree$. Also, let $\sequence{V}=(V_i,i \ge 1)$ be a sequence of independent uniformly random elements of $[n]$ and let $I = \min(i \ge 2: V_i \in \{V_1,\ldots,V_{i-1}\})$ be the index of the first repeated element of $\sequence{V}$. Then 
		$|L|+1 \eqdist \min(I,n)$. 
	\end{prop}
	\begin{proof}
		Write $\sequence{V}=(V_1,\ldots,V_{n-1}) \in_u [n]^{n-1}$. Then $\randomtree=\tree(\sequence{V}) \in_u \trees_n$. Moreover, recalling that the repeated entries of $(V_1,\ldots,V_{n-1})$ are $j(1),\ldots,j(k-1)$ and that $j(k)=n$, we have $\min(I,n)=j(1)$. The first leaf $\ell_1(\randomtree)$ is a child of $V_{i(1)-1}$, so 
		\[
		|\ell_1(\randomtree)|=|V_{j(1)-1}|+1 = j(1)-1=\min(I,n)-1.
		\]
		But since $\randomtree$ is a uniformly random tree, randomly permuting its leaf labels does not change its distribution, so $|\ell_1(\randomtree)|$ has the same distribution as $|L|$ for $L$ a uniformly random leaf of $\randomtree$. 
	\end{proof}
	
	One can also use the version of the bijection for rooted trees with one marked vertex to show the following similar result, whose proof is left to the reader.
	\begin{prop}
		Let $\randomtree \in_u \trees_n$ and let $U \in_u [n]$ be independent of $\randomtree$.  
		Then $|U| \eqdist I-2$, where $I$ is as in Proposition~\ref{prop:leaf_identity}. 
		\label{prop:distance_vtx}
	\end{prop}
	These propositions in particular imply that if $D_n$ is the distance from the root to a uniformly random vertex (or leaf) in $\randomtree_n \in_u \trees_n$, then $n^{-1/2} D_n \convdist R$, where $R$ is Rayleigh distributed: $\p{R \ge x} = e^{-(x^2/2)\I{x \ge 0}}$. 
	It is not hard to build on these statements in order to precisely characterize the joint distribution of the lengths of the branches to the $k$ smallest labeled leaves (using the bijection) or to $k$ uniform vertices (using its extension to trees with $k$ marked vertices). One thereby recovers the asymptotic {\em line-breaking construction} of uniformly random trees, proposed by Aldous \cite{aldous91continuum1}. 
	
	These distributional identities also make it possible to compare the distributions of typical distances in different trees. Define a partial order $\prec$ on type sequences $\sequence{n}=(n_c,c \ge 0)$ with $\sum_{c \ge 0} n_c = n$
	and $\sum_{c\geq 0} c n_c=n-1$
	by the following covering\footnote{
		For a partially ordered set $(\mathcal{P},\prec)$, $y \in\mathcal{P}$ covers $x \in \mathcal{P}$ if $x \prec y$ and for all $z \in \mathcal{P}$, if $x \preceq z \preceq y$ then $z \in \{x,y\}$. 
	}
	relation: $\sequence{m}=(m_c,c \ge 0)$ covers $(n_c,c \ge 0)$ if there are positive integers $a,b$ such that $m_a=n_a+1, m_b=n_b+1$ and $m_0=n_0-1, m_{a+b}=n_{a+b}-1$. (In words, to obtain $\sequence{m}$ from $\sequence{n}$ we replace one vertex with $a+b$ children by two vertices, one with $a$ children and one with $b$ children, and reduce the number of leaves accordingly; then $\sequence{n} \prec \sequence{m}$.)
	
	Given type $\sequence{n}$, $\sequence{m}$ related as in the previous paragraph, we may couple the constructions of random trees with types $\sequence{n}$ and $\sequence{m}$ as follows. Let $\sequence{V}_{\sequence{n}}=(V_1,\ldots,V_{n-1})$ be a uniformly random sequence of elements of $[n]$
	subject to the constraint that for each $c\geq 0$, the number of values from $[n]$
	which occur precisely $c$ times in the sequence is $n_c$.
	Conditionally on $\sequence{V}_{\sequence{n}}$, choose $X \in [n]$ uniformly at random from among those integers in $[n]$ which appear exactly $a+b$ times in $\sequence{V}_{\sequence{n}}$, and independently choose $Y$ uniformly at random from among those integers in $[n]$ which do not appear in $\sequence{V}_{\sequence{n}}$. Choose $a$ of the instances where $X$ appears in $\sequence{V}_{\sequence{n}}$, uniformly at random, and replace each them by the integer $Y$; call the resulting sequence $\sequence{V}_{\sequence{m}}$. Then the trees $\randomtree_{\sequence{n}}$ and $\randomtree_{\sequence{m}}$ corresponding to $\sequence{V}_{\sequence{n}}$ and $\sequence{V}_{\sequence{m}}$ under the bijection are uniformly random elements of $\trees(\sequence{n})$ and $\trees(\sequence{m})$, respectively. Moreover, the index of the first repeated element of $\sequence{V}_{\sequence{n}}$ is at most that of the first repeated element of $\sequence{V}_{\sequence{m}}$, so the distance from the root to the smallest labeled leaf in $\randomtree_{\sequence{n}}$ is at most the corresponding distance in $\randomtree_{\sequence{m}}$. 
	
	It follows from this coupling that for any types $\sequence{n}$ and $\sequence{m}$ with $\sequence{n} \preceq \sequence{m}$, if $\randomtree_{\sequence{n}} \in_u \trees(\sequence{n})$ and $\randomtree_{\sequence{m}} \in_u \trees(\sequence{m})$, and $L_{\sequence{n}}$ and $L_{\sequence{m}}$ are uniformly random leaves of $\randomtree_{\sequence{n}}$ and $\randomtree_{\sequence{m}}$, respectively, then $|L_{\sequence{n}}|$ is stochastically dominated by $|L_{\sequence{m}}|$, by which we mean that $\p{|L_{\sequence{n}}| \le t} \ge \p{|L_{\sequence{m}}| \le t}$ for all $t \in \R$; we denote this relation by 
	$|L_{\sequence{n}}| \preceq_{\mathrm{st}} |L_{\sequence{m}}|$.
	We conjecture that this stochastic relation also holds for the heights of the trees: writing $\hyt(\tree):=\max(|v|,v \mbox{ is a vertex of }\tree)$, then
	\[
	\hyt(\randomtree_{\sequence{n}}) \preceq_{\mathrm{st}} \hyt(\randomtree_{\sequence{m}}),
	\]
	whenever $\randomtree_{\sequence{n}} \in_u \trees(\sequence{n})$ and $\randomtree_{\sequence{m}} \in_u \trees(\sequence{m})$ and $\sequence{n} \preceq \sequence{m}$. 
	
	\smallskip
	
	Some of the most valuable probabilistic consequences of the bijection arise when studying distances in trees with given degrees, using the variant of the bijection that we introduce at the start of Section~\ref{sec:relatedconstr}. 
	In particular, in recent work, \citet{blanc} proves convergence toward inhomogeneous continuum random trees for random combinatorial trees with given types, in great generality, resolving a conjecture from \cite{AldousMiermontPitman} on Lévy trees being a `mixture' of inhomogeneous continuum random trees; he also proves exponential upper tail bounds for the heights of such random trees, which are tight up to constant factors in many cases. In other recent work, \citet{us} use the bijection in order to prove several conjectures from \cite{addarioberryfattrees}, \cite{MR3077536},  \cite{MR2908619} and \cite{McDiarmidScott} regarding the asymptotic behaviour of the height of random trees and tree-like structures; the work \cite{us}  in particular shows that {\em all} random combinatorial trees with $n$ nodes have typical height $O(n^{1/2})$, unless they are extremely ``path-like'', possessing $n-o(n)$ nodes with exactly one child.
	
	The bijection can also be adapted to study the structure of random graphs which are not trees (see \cite[Chapter 1]{abnotes}), and we expect that it will prove useful in proving convergence results in such settings as well; a first example is found in \cite{blancgraphs}. 
	
	\subsection{Growth procedures for rooted trees with a given degree sequence}
	
	In this section, we present a growth procedure for rooted trees with a given degree sequence --- and in particular for $d$-ary trees --- which is built from a version of the inverse bijection 
	presented in the proof of Theorem~\ref{thm:cayley}. To the best of our knowledge, this is the first growth procedure for uniformly random trees with a given degree sequence. Other growth procedures for uniform $d$-ary trees have appeared in the literature, but we believe that the procedure obtained by specializing our general method to $d$-ary trees is the simplest one known so far. We briefly describe the two other growth procedures for uniform $d$-ary trees that we are aware of at the end of this section. 
	
	Fix an abstract set $\{\ell_i,i \ge 1\}$ of leaf labels. 
	Let $\mathbf{d}=(d_1,\dots,d_m)$ be a sequence of positive integers, and let  
	$L_{\mathbf d}=1+\sum_{i=1}^m (d_i-1)\geq 1$. 
	Let $\cT_\mathbf{d}$ denote the set of rooted trees $\tree$ with vertex set $[m]\cup \{\ell_1,\dots, \ell_{L_{\mathbf d}}\}$ such that for each $i\in [m]$, $i$ has $d_i$ children in $\tree$ (and $\ell_1,\ldots,\ell_{L_{\mathbf d}}$ are leaves of $\tree$).  Also, write 
	\[
	\cS_{\mathbf{d}}:=\left\{(v_1,\dots, v_{m + L_{\mathbf d} - 1}):|\{k:v_k=i\}|=d_i\text{ for all }i\in [m]\right\}
	\]
	Then the following modification of the inverse bijection from Theorem~\ref{thm:cayley} gives a bijection between $\cS_{\mathbf d}$ and $\cT_{\mathbf d}$. 
	Fix $\sequence{v}=(v_1,\dots, v_{m + L_{\mathbf d} - 1}) \in \cS_{\mathbf d}$. 
	\begin{tcolorbox}[title=Bijection between $\cS_{\mathbf d}$ and $\cT_{\mathbf d}$.]
		\begin{itemize}
			\item Let $j(0)=1$, let $j(1)<j(2)<\dots<j(L_{\mathbf d}-1)$ be the locations of the repeats
			of the sequence $\sequence{v}$, and let $j(L_{\mathbf d})=m+L_{\mathbf d}$.
			\item For $i=1,\dots, L_{\mathbf d}$, let $P_i$ be the path
			$(v_{j(i-1)}, \dots, v_{j(i)-1}, \ell_i)$.
			\item Let $\tree(\sequence{v}) \in \cT_{\mathbf d}$ have root $v_1$ and edge set given by the union of the edges of the paths
			$P_1, P_2, \dots, P_{L_{\mathrm d}}$.
		\end{itemize}
	\end{tcolorbox}
	In the above bijection, the path $P_i$ ends with the leaf $\ell_i$. However, any other fixed ordering of $\ell_1,\ldots,\ell_{L_{\mathbf d}}$ would also yield a bijective correspondence. In fact, this is even true if the choice of leaf ordering depends on the sequence $\sequence{v}$, provided that the leaf-ordering rule still has the property that different sequences result in different trees. We exploit this flexibility below, in order to design a simple growth procedure. 
	
	Write $\preceq$ for the total ordering of the vertices of $\tree(\sequence{v})$ which is the order the vertices first appear along the paths $P_1,\ldots,P_{L_{\mathrm d}}$. 
	
	For $d\in\N$ and $\mathbf{d}':=(d_1,\dots,d_m,d)$, note that we can construct an element of $\cS_{\mathbf{d}'}$  starting from an element $S_\mathbf{d} \in \cS_{\mathbf{d}}$ by inserting $d$ repeats of integer $m+1$. (Later, we also write $\mathbf{d}'=(d_1,\ldots,d_m,d_{m+1})$, so $d_{m+1}=d$.) This corresponds to choosing $S_\mathbf{d} \in \cS_{\mathbf{d}}$ and a multiset of size $d$ with elements in $[m + L_{\mathbf d}]$, or equivalently, $T_{\mathbf d} \in \cT_{\mathbf d}$ and a multiset of $d$ vertices of $T_{\mathbf d}$, where we use the convention that $i\in [m+ L_{\mathbf d}]$ corresponds to the $i^{th}$ vertex in the total ordering $\preceq$ of the vertices of $T_{\mathbf d}$ defined above. 
	The following growth procedure provides a way to sample a uniform random element of $\cT_\mathbf{d'}$ given $\randomtree_{\mathbf d} \in_u \cT_{\mathbf{d}}$ and an independent uniform random multiset of size $d$ of vertices of $\randomtree_{\mathbf d}$. 
	It yields the same shape of tree that one would obtain starting from $\randomtree_{\mathbf d}$ by first transforming the sequence $S_d$ corresponding to $\randomtree_{\mathbf d}$ under the above bijection as just described, then applying the bijection to the resulting sequence. However, it yields a different labelling of the leaves. 
	
	\begin{tcolorbox}[title = Constructing a tree $T_{\mathbf d'}\in \cT_\mathbf{d'}$ from a tree $T_{\mathbf d}\in \cT_\mathbf{d}$ and a multiset of $d$ vertices of $T_{\mathbf d}$]
		Denote the multiset as $\{w_1,\dots,w_d\}$, where $w_1 \preceq w_2 \preceq \ldots \preceq w_d$.
		Then $T_{\mathbf{d}'}$ is obtained from $T_{\mathbf{d}}$ by adding vertices $m+1, \ell_{L_{\mathbf d}+1}, \ldots, \ell_{L_{\mathbf d}+d-1}$ and modifying the edge-set as follows.
		\begin{enumerate}
			\item If $w_1$ was the root of $T_{\mathbf d}$, connect $m+1$ to $w_1$ and re-root the tree at $m+1$. Otherwise, replace the edge $vw_1$ that connects $w_1$ to its parent $v$ with two edges $v(m+1)$ and $(m+1)w_1$.
			\item For $j = 2,\ldots,d$,
			\begin{enumerate}
				\item if $w_j = w_{j-1}$, add an edge from $m+1$ to $\ell_{L_{\mathbf d}+j-1}$, 
				\item if $w_j \neq w_{j-1}$, remove the edge from $w_j$ to its parent $v$, then add edges from $v$ to $\ell_{L_{\mathbf d}+j-1}$ and from $m+1$ to $w_j$. 
			\end{enumerate} 
		\end{enumerate}
	\end{tcolorbox}
	\begin{figure}[hbt]
		\begin{centering}\includegraphics[width=0.9\textwidth]{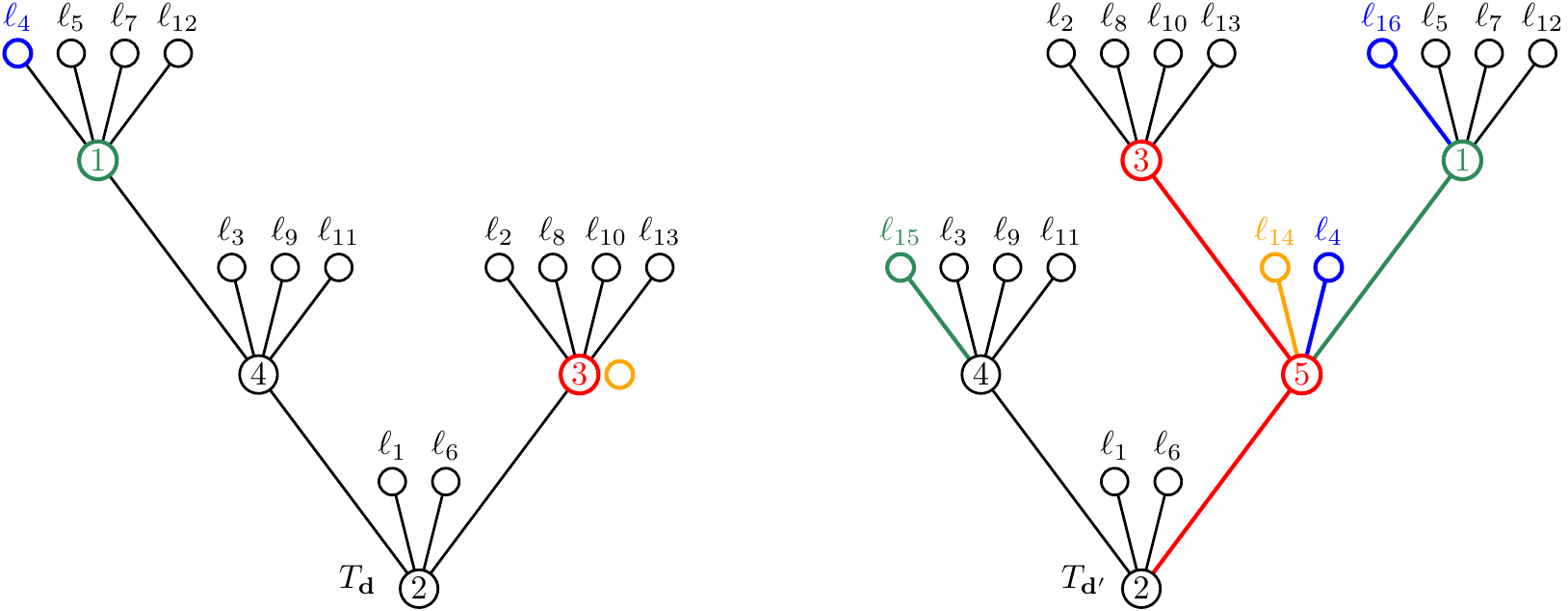}
			\caption{The growth procedure is used to add one internal vertex to the quarternary tree $T_{\mathbf{d}}$. Here, $(w_1,w_2,w_3,w_4)=(3,3,1,\ell_4)$.
				\label{fig:growth_4ary}}
		\end{centering}
	\end{figure}
	
	An example is provided in Figure \ref{fig:growth_4ary}. In the example, the coding sequence of $T_{\mathbf{d}}$ under the bijection presented above is $(2,2,3,2,4,4,1,1,2,1,3,4,3,4,1,3)$.
	The total ordering $\preceq$ of vertices of $T_{\mathbf d}$ is 
	\[
	2 \prec \ell_1 \prec 3 \prec \ell_2 \prec 4 \prec \ell_3 \prec 1 \prec \ell_4 \prec \ell_5 \prec \ell_6  \prec \ldots \prec \ell_{13}\, .
	\]
	The chosen multiset $(3,3,1,\ell_4)$ of vertices of $T_{\mathbf d}$ corresponds to the multiset $(3,3,7,8)$, of elements of $[m+L_{\mathbf{d}}]$, and thence to transforming the coding sequence into 
	\[
	(2,2,\mathbf{5},\mathbf{5},3,2,4,4,\mathbf{5},1,\mathbf{5},1,2,1,3,4,3,4,1,3).
	\]
	The new sequence gives $T_{\mathbf{d}'}$ if we modify the bijection so that the new leaves $\ell_{14},
	\ell_{15},\ell_{16}$ are inserted when the repeats of $5$ occur, and the other leaves are inserted according to their original ordering. In general, the growth procedure yields a modified bijection which reads as follows. With $\mathbf{d}'$ as above, fix $\sequence{v}'=(v_1',\dots, v'_{m+L_{\mathbf{d}'}}) \in \cS_{\mathbf{d}'}$, and note that $L_{\mathbf{d}'}=L_{\mathbf d}+d-1$. 
	
	\begin{tcolorbox}[title=Modified bijection between $\cS_{\mathbf{d}'}$ and $\cT_{\mathbf{d}'}$.]
		\begin{itemize}
			\item Let $j(0)=1$, let $j(1)<j(2)<\dots<j(L_{\mathbf{d}'}-1)$ be the locations of the repeats of the sequence $\sequence{v}'$, and let $j(L_{\mathbf{d}'})=m+1+L_{\mathbf{d}'}$.
			\item Reorder $\ell_1,\ldots,\ell_{L_{\mathbf{d}'}}$ as $\hat{\ell}_1,\ldots,\hat{\ell}_{L_{\mathbf{d}'}}$ as follows. For $1 \le i \le L_{\mathbf{d}'}-1$: 
			\begin{itemize} 
				\item If $v'_{j(i)}$ is the $(k+1)$'st appearance of $(m+1)$ then let $\hat{\ell}_i=\ell_{L_{\mathbf{d}}+k}$. 
				\item If $v'_{j(i)}$ is the $k$'th repeated entry which is not equal to $(m+1)$ then let $\hat{\ell}_i=\ell_k$. \end{itemize}
			Let $\hat{\ell}_{L_{\mathbf{d}'}} = \ell_{L_{\mathbf{d}}}$. 
			\item For $i=1,\dots, L_{\mathbf{d}'}$, let $P_i'$ be the path
			$(v_{j(i-1)}, \dots, v_{j(i)-1}, \hat{\ell}_i)$.
			\item Let $\tree'(\sequence{v}') \in \cT_{\mathbf d'}$ have root $v'_1$ and edge set given by the union of the edges of the paths
			$P'_1, P'_2, \dots, P'_{L_{\mathbf{d}'}}$.
		\end{itemize}
	\end{tcolorbox}
	We now describe the inverse of the modified bijection, which takes as input a tree 
	$\tree' \in \cT_{\mathbf{d}'}$ 
	--- i.e., a tree $\tree'$ with vertex set $[m+1]\cup \{\ell_1,\dots, \ell_{L_{\mathbf{d}'}}\}$ such that for each $i\in [m+1]$, $i$ has $d_i$ children, and $\ell_1,\ldots,\ell_{L_{\mathbf{d}'}}$ are leaves ---  and outputs a sequence $\sequence{v}'=(v'_1,\ldots,v'_{m+L_{\mathbf{d}'}}) \in \cS_{\mathbf{d}'}$. We leave it to the reader to verify that the two procedures are indeed inverses. This justifies the fact that the above growth procedure 
	takes the uniform distribution on $\cT_{\mathbf{d}}$ to the uniform distribution on $\cT_{\mathbf{d}'}$.
	
	\begin{tcolorbox}[title=Inverse of the modified bijection]
		\begin{itemize}
			\item  Let $\sequence{v}'_{m+L_{\mathbf{d}'}}$ be equal to the parent of $\ell_{L_{\mathbf{d}}}$ in $\tree'$. 
			Say that $\ell_{L_{\mathbf d}}$ is used, and that $\ell_1,\dots,\ell_{L_{\mathbf{d}}-1}$ and  $\ell_{L_{\mathbf{d}}+1},\dots,\ell_{L_{\mathbf{d}'}}$ are unused.
			\item For 
			$j=m+L_{\mathbf{d}'}-1,\ldots,2,1$: 
			\begin{itemize}
				\item If the number of occurrences of $\sequence{v}'_{j+1}$ in $\big(\sequence{v}'_{j+1},\dots, \sequence{v}'_{m+L_{\mathbf{d}'}}\big)$ is equal to the number of children of $\sequence{v}'_{j+1}$ in $\tree'$, then let $\sequence{v}'_{j}$ be equal to the parent of $\sequence{v}'_{j+1}$ in $\tree'$. 
				\item Otherwise, define $\sequence{v}'_{j}$ as follows. 
				\begin{itemize}
					\item If $\sequence{v}'_{j+1}=m+1$, let $i^*$ be the maximum $L_{\mathbf{d}}+1\leq i\leq L_{\mathbf{d}'}$ such that $\ell_i$ is unused.
					\item If $\sequence{v}'_{j+1}\ne m+1$, let $i^*$ be the maximum $1\leq i\leq L_{\mathbf{d}}-1$ such that $\ell_i$ is unused. 
				\end{itemize}
				Let $\sequence{v}'_{j}$ be equal to the parent of $\ell_{i^*}$ in $\tree'$. Say that $\ell_{i^*}$ is used.
			\end{itemize}
		\end{itemize}
	\end{tcolorbox}
	
	We conclude the paper by justifying the assertion from Section~\ref{sec:intro} that the growth procedure yields an algorithm for sampling a sequence of random $d$-ary leaf-labeled trees. For this, let $\mathbf{d}^{(m)}=(d,d,\ldots,d)$ have length $m$. Starting from the unique tree $\randomtree^{(1)} \in_u \cT_{\mathbf{d}^{(1)}}$, for each $m \ge 1$ let $\randomtree^{(m+1)}$ be constructed from $\randomtree^{(m)}$ according to the above procedure. Now let $\mathrm{T}_m$ be obtained from $\randomtree^{(m)}$ by unlabeling the non-leaf vertices; then $\mathrm{T}_m$ is a uniformly random leaf-labeled $d$-ary tree with $m$ internal nodes, and $(\mathrm{T}_m,m \ge 1)$ is the random sequence alluded to in the introduction. (In fact, given a multiset of size $d$ of vertices of $\randomtree^{(m)}$, its ordering, as defined in the construction of $\randomtree^{(m+1)}$ from  $\randomtree^{(m)}$,  does not depend on the labels of the internal vertices. Therefore, it is straightforward to define the pushforward of the growth procedure directly on leaf-labeled trees.)
	Then, given  $\mathrm{T}_{m+1}$, the non-leaf vertex that was added to obtain $\mathrm{T}_{m+1}$ from $\mathrm{T}_m$ is a uniformly random non-leaf vertex of $\mathrm{T}_{m+1}$. In the binary case, this distinguishes the procedure from R\'emy's algorithm, for which, given the shape of $\mathrm{T}_{m+1}$, the non-leaf vertex that was added is instead the parent of a uniformly random leaf of $\mathrm{T}_{m+1}$.
	
	The growth procedure for planar $d$-ary trees described in \cite{Marckert} is most comparable to ours. Like our procedure, it is based on moving subtrees in the tree, but while we sample a uniform multiset of vertices, the author of \cite{Marckert} samples a uniform set of edges (possibly including `fake' edges called \emph{buds}) and an independently chosen uniform element in $[d]$ to define a step of the growth procedure. In the procedure in \cite{Marckert}, the subtrees always get attached to a new vertex below the old root, while our attachment point is distributed as a uniformly random non-leaf vertex. 
	The growth procedure in \cite{Marckert} is arguably more complex, both to describe and to verify, but a single growth step takes constant time, so the procedure yields a linear time algorithm for growing uniform $d$-ary trees. In our procedure, the part of a growth step that we do not see how to execute in constant time is determining the order under $\preceq$ of  the elements in the selected multiset.
	
	In \cite{LuczakWinkler}, the authors show the existence of a local growth procedure for planar $d$-ary trees, in which a random leaf is replaced by a new vertex with $d$ children. However, the probabilities with which the leaves are chosen are implicitly defined, which makes the growth procedure more challenging to implement in practice.
	
	\section{Acknowledgements}
	We thank Arthur Blanc-R\'enaudie and Nicolas Broutin for useful discussions, Adrien Segovia for pointing out the reference \cite{FoataFuchs}, and two anonymous referees for useful comments. 
	During the preparation of this work MM was supported by EPSRC Fellowship EP/N004833/1.
	
	\bibliographystyle{plainnat}
	\bibliography{cayley}
	
	\appendix
	
\end{document}